\documentclass[conf]{new-aiaa}
\usepackage[utf8]{inputenc}
\usepackage[english]{babel}

\usepackage{amsmath}

\usepackage{amsthm}

\usepackage{graphicx} 	
\graphicspath{{img/}}
\usepackage{subfig}
\usepackage{booktabs} 	
\usepackage{cleveref} 	
\crefformat{equation}{(#2#1#3)}
\crefmultiformat{equation}{(#2#1#3)}
{ and~(#2#1#3)}{, (#2#1#3)}{ and~(#2#1#3)}
\crefrangeformat{equation}{(#3#1#4) to~(#5#2#6)}

\usepackage{mathrsfs}
\usepackage{siunitx}
\usepackage{longtable,tabularx}
\usepackage{float}
\usepackage{xcolor}

\usepackage{tikz}
\usepackage{tikzscale}
\usetikzlibrary{backgrounds}
\usepackage{pgfplots}
\usetikzlibrary{external}
\pgfplotsset{compat=newest}
\usepgfplotslibrary{groupplots}
\usepgfplotslibrary{polar}
\usepgfplotslibrary{smithchart}
\usepgfplotslibrary{statistics}
\usepgfplotslibrary{dateplot}
\usepgfplotslibrary{fillbetween}

\newtheorem{thm}{Theorem}

\newtheorem*{rmk}{Remark}

\DeclareMathOperator*{\argmin}{argmin}

\title{An Optimal Engagement Zone Avoidance Scenario in 2-D
}

\author{
	Isaac E. Weintraub\footnote{Electronics Engineer, Controls Science Center, Air Force Research Laboratory, WPAFB, OH 45433 AIAA Senior Member.},
	Alexander Von Moll\footnote{Electronics Engineer, Controls Science Center of Excellence, Air Force Research Laboratory, WPAFB, OH 45433},
	Christian Carrizales\footnote{Aerospace Engineer, Aerospace Systems Directorate, Air Force Research Laboratory, WPAFB, OH 45433},
	Nicholas Hanlon\footnote{Aerospace Engineer, Aerospace Systems Directorate, Air Force Research Laboratory, WPAFB, OH 45433}, and
	Zachariah Fuchs\footnote{Assistant Professor, Electrical Engineering, University of Cincinnati, Cincinnati, OH 45221}
}

\begin{document}

\maketitle

\begin{abstract}
In this paper, an optimal control problem is considered where a target vehicle aims to reach a desired location in minimum time while avoiding a dynamic engagement zone. Using simple motion, four potential approaches are considered. First, the min-time strategy which ignores the engagement zone is posed and solved. Second, the min-time strategy which avoids the engagement zone entirely is considered. Third, the min-time strategy which allows for some time in the engagement zone; but, still strives to stay away from the center of the engagement zone is posed. Lastly, a fixed final-time strategy is considered, wherein the target tries to avoid the engagement zone; but, is required to arrive at the desired location at a specific time. Using a nonlinear program solver, the optimal strategies are numerically solved. From the results of the numeric solutions, the optimal strategies are discussed and comparisons are drawn.

\end{abstract}

\section{Nomenclature}

{\renewcommand\arraystretch{1.0}
\noindent\begin{longtable*}{@{}l @{\quad=\quad} l@{}}
$c(\cdot)$ & path constraint function\\
$\mathscr{C}$ & termination set \\
$d$ & vehicle range [DU] \\
$F$ & final location \\
$\mathbf{f}$ & dynamics / equations of motion\\
$g_{EZ}$& penalty function for the engagement zone\\
$J_i$ & objective cost functional for scenario $i$\\
$k_{EZ}$ & objective cost weight\\
$l(\cdot)$ & running cost \\
$v_T$ & target velocity [DU/TU]\\
$\mathscr{H}$ & Hamiltonian\\
$\mathbf{p}$ & costate vector\\
$p_i$ & costate corresponding to a state $i$ \\
$R_{\mathrm{min}}$ & weapon engagement zone minimum range [DU] \\
$R_{\mathrm{max}}$ & weapon engagement zone maximum range [DU] \\
$T$ & target location\\
$t$ & time [TU] \\
$t_f$ & final time [TU] \\
$t_\mathrm{go}$ & required final time [TU] \\
$u$ & control [rad] \\
$W$ & engagement zone origin \\
$\mathcal{W}$ & engagement zone region\\
$\mathbf{x}$  & state of the scenario \\
$\mathbf{x}_0$  & initial state of the scenario \\
$\mathbf{x}_f$  & final state of the scenario \\
$\hat{x}$ & Cartesian fixed frame x-direction\\
$x_T$ & target vehicle's x-location [DU]\\
$\hat{y}$ & Cartesian fixed frame y-direction\\
$y_T$& target vehicle's y-location [DU]\\
$\lambda$ & line of sight angle [rad] \\
$\mu$ & Lagrange variable\\
$\psi_T$ & target vehicle heading\\
$\rho$ & weapon engagement zone range [DU] \\
$\rho_\mathrm{max}$ & maximum weapon engagement zone range [DU]\\
$\sigma_i$ & variable used for simplifying equations\\
$\tau$ & half-angle variable\\
$\theta$ & engagement zone angle [rad]\\
$\xi$ & aspect angle [rad] \\
\end{longtable*}}

\section{Introduction}

\lettrine{R}{eal} world applications for trajectory optimization typically entail constraints which are functions of the state variables, control variables, or both.
Consider, for example, such applications as Urban Air Mobility (UAM) and Unmanned Traffic Management (UTM) \cite{ramee2021development,prevot2016uas, faa2018unmanned} in which individual vehicles are constrained in the amount of control effort they can apply; but more interestingly, they are constrained in what physical space they can occupy at what times.
The latter challenge has been approached in a variety of ways, e.g., via collision avoidance algorithms~\cite{manyam2021quadratic}, interactive trajectory planning~\cite{davis2012development, casbeer2006cooperative, rasmussen2017practical}, or scheduling algorithms~\cite{sama2018coordination}.
In this work, we abstract the notion of what space \textit{cannot} be occupied by a particular vehicle at a particular time as a dynamic keepout zone or Engagement Zone (EZ).
We treat the EZ in a variety of ways: 1) as a hard constraint -- meaning the vehicle must remain outside it at all times, 2) as a soft constraint, and 3) as the objective cost functional for a fixed final time scenario.
Then, the trajectory optimization for a single vehicle, the target vehicle, is addressed by means of optimal control.

Optimal control and flying vehicles have long gone hand-in-hand, e.g., all the way back to the famous Goddard rocket \cite{maurer1976numerical,goddard1919method}.
Although many of the theoretical aspects have been addressed (c.f.~\cite{bryson1975applied}), constraints pose some practical difficulties in applying the so-called indirect methods of optimal control, which are based on the first order necessary conditions for optimality.
In particular, it is often necessary to concatenate a series of arcs (or sub-trajectories) wherein each arc is assumed to be unconstrained ($U$) or constrained ($C$).
The numerical solution procedure is referred to as multiple shooting.
Additional complexities arise when there are multiple constraints and the sequence of constraint activation and/or deactivation is unknown.
Alternatively, direct optimization methods such as collocation, or pseudo-spectral methods (PSM)~\cite{fahroo2005pseudospectral, darby2009state}, can circumvent some of the practical, implementation issues~\cite{kelly2017introduction, weintraub2020direct}.
As will be shown, PSM readily handles the different interpretations of the EZ, mentioned above.
Nonetheless, we pursue both approaches here to develop an understanding of the essential features of the solution.

Our notion of the EZ extends also to the realm of adversarial scenarios.
In this setting, the EZ could represent undesired collision, observation or exposure to an enemy~\cite{weintraub2020maximum,weintraub2021engagement}, or, more generally, states from which the target vehicle can be guaranteed to lose in a differential game sense.
The latter is referred to as the \textit{Game of Kind}, and the boundary of the EZ thus may represent the associated barrier surface~\cite{isaacs1965differential}.
Unlike in the traditional differential game setting, any notion of adversarial behavior has been abstracted into the EZ, and the target vehicle seeks to accomplish some external goal.
In this case, the target vehicle must maneuver from a pre-specified initial state to a goal state.
A similar scenario was investigated in~\cite{fuchs2017generalized} under the name of optimal constrained retreat (OCR).
There, a vehicle sought to maneuver past a static defender to a pre-specified retreat zone in minimum time subject to the constraint that it never became advantageous (optimal) to engage the defender in lieu of retreat.
OCR was also treated in~\cite{vonmoll2020optimal} for the case where the defender was a turn-rate-constrained turret.
A significant difference in the present work is the EZ considered here is not only dynamic (in the sense that it is time-varying), but also a function of the target vehicle's control input.
This introduces a coupling between the constraint and the vehicle's control which has not yet been considered.
In addition to providing solutions to the trajectory optimization problems under each of the EZ interpretations, an important contribution of this work is the closed-form analytic expression for the optimal control when the (hard) EZ constraint is active.

The remainder of the paper is organized as follows.
Section~\ref{sec:optimal-control-problem} specifies four optimal control problems: A) (unconstrained) minimum time, B) minimum time with EZ as a hard constraint, C) minimum time with EZ as a soft constraint, and D) minimum EZ violation subject to a desired arrival time.
Section~\ref{sec:optimal-target-vehicle-strategy} presents the analysis of the optimal control for each of the scenarios.
Section~\ref{sec:results} contains the solution trajectories for a particular initial and goal state.
Finally, Section~\ref{sec:conclusions} concludes the paper by summarizing the main results and identifying areas for further research.

\section{Optimal Control Problem}
\label{sec:optimal-control-problem}
Consider the optimal engagement zone avoidance scenario wherein a target vehicle $(T)$ aims to reach a desired location ($F$) in minimum time; while, at the same time, strives to avoid the engagement zone entirely where possible, and stay farthest from the engagement zone's origin, otherwise. The target vehicles state ($\mathbf{x}$) is defined by its pair of Cartesian coordinates, $(x_T,y_T)^\top \in \mathbb{R}^2$. The holonomic target vehicle controls its heading $(\psi_T)$ directly; the admissible control is $u(t) = \psi_T \in (-\pi,\pi]$. The target is assumed to have constant strictly-positive speed, $v_T>0$. With all these definitions, the complete state of the system is as follows:
\begin{equation}
	\mathbf{x} = [x_T, y_T]^\top \in \mathbb{R}^2.
\end{equation}
The state dynamics are comprised of the target's motion and are governed by the following ordinary nonlinear differential equations:
\begin{equation}
	\begin{aligned}
		\dot{x}_T(t) &= v_T \cos(\psi_T(t)),\\
		\dot{y}_T(t) &= v_T \sin(\psi_T(t)).
	\end{aligned}
	\label{eq:EOM}
\end{equation}
The initial location of the target is $\mathbf{x}_0 = \mathbf{x}(t_0) = [x_{T0},y_{T0}]^\top$. The desired final location of the target, its goal position, is $\mathbf{x}_f = \mathbf{x}(t_f) = [x_{Tf},y_{Tf}]^\top$. 

The engagement zone (EZ) model is dynamic in nature and is a function of the line of sight angle from the engagement zone's origin to the target, $\lambda \in [-\pi,\pi)$, the relative bearing from the target to the engagement zone's origin (EZO), $\xi \in [0,2\pi)$ and $\theta \in [0,2\pi)$, the angle from the EZO to any arbitrary location in the state space, In polar form, the general EZ is as follows:
\begin{equation}
	\rho(\theta;\lambda,\xi) = \left(\left( \dfrac{\cos \xi + 1}{2} \right)\left(R_{\mathrm{max}} -R_{\mathrm{min}}\right) + R_{\mathrm{min}} \right)\dfrac{1}{2}\left( 1 + \sin \left(\dfrac{\pi}{2} - \lambda + \theta\right)\right).
	\label{eq:generalEZ}
\end{equation} 
As an assumption, the minimum range of the EZ is selected to be $R_\mathrm{min} \equiv 0$, and therefore the general EZ model in \cref{eq:generalEZ}, as used throughout this paper is,
\begin{equation}
	\rho(\theta;\lambda,\xi) = \dfrac{R_\mathrm{max}}{2} \left( \dfrac{\cos \xi + 1}{2} \right)  \left( 1 + \sin \left(\dfrac{\pi}{2} - \lambda + \theta\right)\right).
	\label{eq:EZModel}
\end{equation}
The EZ model is defined in \cref{eq:EZModel} so that when the EZ angle, $\theta$, is equal to the line of sight angle $\lambda$
In general, the cardiod-shaped EZ could have any orientation and could be dynamic (i.e., $\theta \equiv \theta(t)$) or even explicitly controlled in the adversarial case.
By setting $\theta(t) = \lambda(t)~\forall t$, we are actually solving for the worst-case for the target vehicle.
That is, a trajectory which avoids the EZ under $\theta = \lambda$ is guaranteed to avoid the EZ in the general $\theta$ case.
The range of the EZ thus corresponds to the maximum of~\eqref{eq:EZModel},
\begin{equation}
	\rho_{\mathrm{max}}(\xi) = \frac{R_{\mathrm{max}}}{2} \left(\cos \xi + 1\right),
	\label{eq:EZ}
\end{equation}
where $R_{\mathrm{max}} > 0$ is the maximum possible range of the engagement zone and $\rho_\mathrm{max}$ is the corresponding maximum engagement zone range from its origin ($W$) as a function of the bearing angle, $\xi$. Taking the absolute value of the bearing angle, $\xi$, the aspect angle of the target vehicle is $|\xi|$. The aspect angle and line of sight angles are given as a function of the state and control are as follows:
\begin{equation}
	\xi(\mathbf{x}(t),\psi_T(t)) = \psi_T(t) - \lambda(\mathbf{x}(t)) - \pi,
	\label{eq:aspectAngle}
\end{equation}
\begin{equation}
	\lambda(\mathbf{x}(t)) = \tan^{-1}(y_T/x_T).
	\label{eq:LOS}
\end{equation}
The distance, $d$, is defined as the instantaneous range of the target vehicle from the engagement zone's origin ($W$):
\begin{equation}
	d(t) = \sqrt{x_T^2(t)+y_T^2(t)}
	\label{eq:distance}
\end{equation}
Final time is defined as the instant when the target, $T$, is collocated at the desired final location,
\begin{equation}
	\mathscr{C} = \lbrace \mathbf{x}(t) | (x_T(t) - x_{Tf})^2 + (y_T(t) - y_{Tf})^2 = 0 ,\; t = t_f \rbrace
	\label{eq:terminalManifold}
\end{equation}

In this paper, four optimal control problems are investigated: To determine the target's strategy which A) reaches the final location in minimum time regardless of the engagement zone, B) avoids the engagement zone completely and reaches the desired final location in minimum time, C) reaches final location in minimum time and distances itself from the engagement zone's origin, D) arrives at the final location at a specified time while distancing itself from the engagement zone's origin. A figure which depicts the EZ avoidance scenarios is shown in \Cref{fig:geometry}. 
\begin{figure}[H]
    \centering
    \includegraphics[width=2.5in]{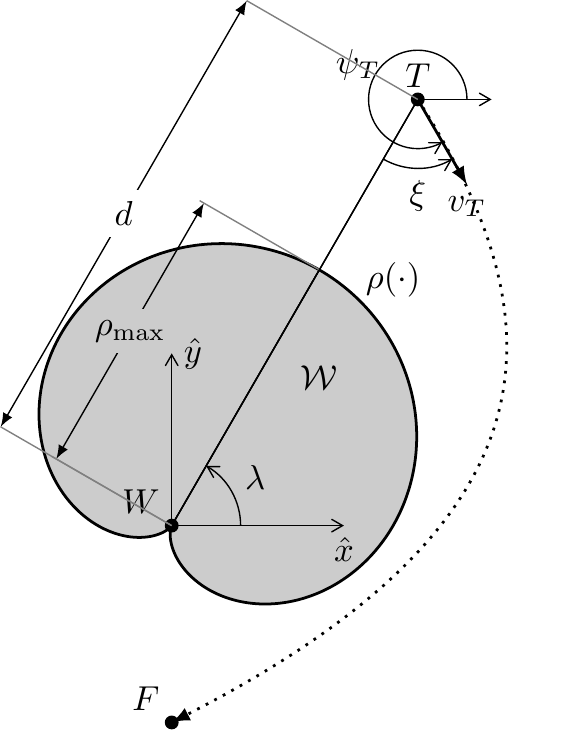}
    \caption{The engagement zone avoidance scenario consists of a maneuverable target vehicle ($T$) which aims at reaching a location ($F$) in minimum time while avoiding the engagement zone ($\mathcal{W}$).}
    \label{fig:geometry}
\end{figure}

\subsection{Min-Time Arrival}
In this scenario, the objective is to arrive at the desired final location in minimum time, regardless of the EZ. This provides a bound on the minimum time. The optimal heading for this scenario is obtained from minimization of the objective function,
\begin{equation}
	u_A^* = \argmin_{\psi_T} J_A = \int_0^{t_f}1 \mathop{\mathrm{d}t} = t_f.
	\label{eq:ObjectiveA}
\end{equation}
For this scenario, the time where the target arrives at the desired final location is denoted $t_{f,A}$.

\subsection{Min-Time Arrival while Avoiding the EZ}
In this scenario, the objective is to arrive at the desired final location in minimum time while avoiding the engagement zone for the entire trajectory. In order to consider this optimal control problem the optimal control is one which satisfies the following:
\begin{equation}
	u^*_B(t) = \argmin_{\psi_T} J_B = \int_0^{t_f} 1 \; \mathrm{d}t = t_f,
	\label{eq:ObjectiveB}
\end{equation}
subject to a path constraint (inequality constraint),
\begin{equation}
	c(\mathbf{x}(t),u(t),t) = \rho_\mathrm{max}(\lambda(x_T,y_T)),\xi(x_T,y_T,\psi_T)) - d(x_T,y_T) \leq 0  \; \forall \; t \in [0,t_f].
	\label{eq:ineqConstr1}
\end{equation}
For this scenario, the time where the target arrives at the desired final location is denoted $t_{f,B}$.

\subsection{Min-Time Arrival while Distancing Itself from the EZO}
In this scenario, the objective is to arrive at the desired final location in minimum time while avoiding the engagement zone's origin. Rather than force the target to completely avoid the EZ, the target is allowed to enter the EZ (while incurring some penalty) to encourage a faster arrival. The penalty function is $g_{EZ}(\mathbf{x}(t),u(t),t)$ and is as follows:
\begin{equation}
	g_{EZ}(\mathbf{x}(t),u(t)) = \begin{cases}
		\dfrac{\rho_\mathrm{max}(\xi(x_T(t),y_T(t),\psi_T(t)))}{d(x_T(t),y_T(t))} - 1 & d \leq \rho \\
		 0 	& d > \rho.
	\end{cases}
	\label{eq:penaltyFunction}
\end{equation}
The penalty function as described in \cref{eq:penaltyFunction} provides a penalty of increasing magnitude as the target gets closer to the EZO (dictated by the first case where $d \leq \rho$). Outside the EZ, no penalty is incurred (dictated by the second case where $d > \rho$). In \Cref{fig:penaltyFunction} the penalty associated with entering the EZ is illustrated. 
\begin{figure}[htb]
	\centering
	\includegraphics[width=5in]{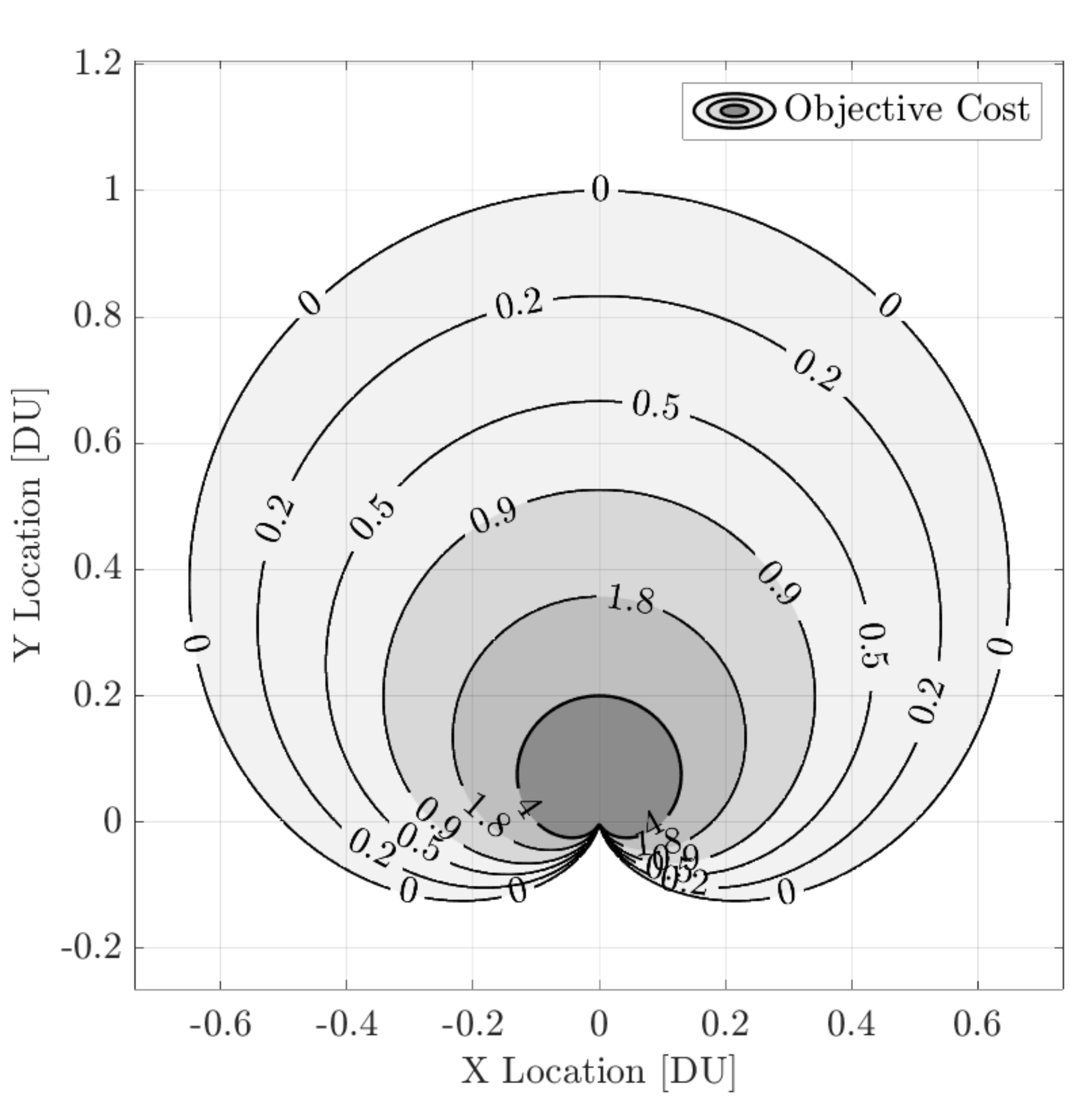}
	\caption{The objective cost that penalizes being closer to the engagement zone's origin.}
	\label{fig:penaltyFunction}
\end{figure}
Contour lines of fixed-value are plotted in the state space. The case where the contour line of value $0$ represents the EZ's border -- no cost is incurred if the target takes a trajectory along the EZ's border.

For this scenario, the objective cost is
\begin{equation}
	u^*_C(t) = \argmin_{\psi_T} J_C = t_f + k_{EZ} \int_0^{t_f} g_{EZ}(\mathbf{x}(t),u(t)) \mathrm{d}t.
	\label{eq:ObjectiveC}
\end{equation}
The objective cost penalizes time inside the EZ. If the target vehicle does not go inside the EZ, the objective cost from \cref{eq:ObjectiveB} is the same as the objective cost in \cref{eq:ObjectiveB}. A weight $k_{EZ}$ is applied to tune the penalty from arrival time and the penalty from taking a trajectory close to the EZO. It should be noted that, theoretically, a gain of $k_{EZ}=0$ is identical to Scenario A. As the gain, $k_{EZ} \rightarrow \infty$, trajectories would remain outside the engagement zone entirely and the outcome is the same as Scenario B.

\subsection{Specified Time of Arrival while Distancing Itself from the EZO}
In this scenario, the objective is to arrive at the desired final location at a specific time. Obviously, the arrival time must be greater than or equal to the minimum time trajectory from Scenario B. Also, this scenario is only investigated if the required arrival time is less than the time from Scenario A (provided the initial and final location of the target are the same). The required arrival time, $t_{\mathrm{go}}$ is selected over the interval,
\begin{equation}
	t_{\mathrm{go}} \in \left[t_{f,A},t_{f,B}\right).
\end{equation}
If the arrival time is strictly greater than the min-time trajectory in Scenario B, the target strives to minimize its range from the WEZ's origin. The objective is to find the control which satisfies \cref{eq:ObjectiveB}, subject to the  boundary condition,
\begin{equation}
	t_f - t_{\mathrm{go}} = 0.
\end{equation}
Therefore, the objective cost for this scenario is
\begin{equation}
	u^*_D(t) = \argmin_{\psi_T} J_D =  \int_0^{t_f} g_{EZ}(\mathbf{x}(t),u(t)) \mathrm{d}t.
	\label{eq:ObjectiveD}
\end{equation}
In \Cref{tab:OCPs}, each of the optimal control problems are summarized.
\begin{table}[H]
\centering
\caption{Optimal Control Problems}
	\begin{tabular}{l c c c c}
		\toprule
		\textbf{Parameter} & \textbf{Scenario A} & \textbf{Scenario B} & \textbf{Scenario C} & \textbf{Scenario D} \\
		\toprule
		Initial State   & $\mathbf{x}_0 = T$ & $\mathbf{x}_0 = T$ & $\mathbf{x}_0 = T$ & $\mathbf{x}_0 = T$ \\
		Terminal State  & $\mathbf{x}_f = F$ & $\mathbf{x}_f = F$ & $\mathbf{x}_f = F$ & $\mathbf{x}_f = F$\\
		Terminal Time   & free & free & free & fixed \\ 
		Objective       & $J_A = t_f $ & $J_B = t_f $ & $J_C = t_f + k_{EZ}\int_0^{t_f} g_{EZ}(\mathbf{x}(t),u(t)) \mathrm{d}t$ & $J_D = \int_0^{t_f} g_{EZ}(\mathbf{x}(t),u(t)) \mathrm{d}t$\\
		Path Constraint & None & $\rho(\cdot) - d(\cdot) \leq 0$ & None & None
	\end{tabular}
	\label{tab:OCPs}
\end{table}

For each of the four scenarios, the initial and final state are fixed and known. In Scenarios A B, and C, the final time is free and in Scenario D, the final time is fixed and known. The objective for Scenario A and B is to minimize the time only; Scenario C mixed (time and range from EZO when the target is inside the EZ), and Scenario D is one which penalizes the range from the EZO when the target is inside the EZ. Scenario B restricts trajectories to those that stay outside the EZ while Scenarios A, C, and D allow the target to enter the EZ; for Scenarios C and D a penalty is applied for occupying the EZ.

\section{Optimal target Vehicle Strategy}
\label{sec:optimal-target-vehicle-strategy}
In this paper, four optimal control problems are compared for the purpose of reaching a desired final location, $F$, in minimum time. \Cref{tab:OCPs} provides a summary of the four considered scenarios. In this section, the four scenarios are posed as optimal control problems.  

\subsection{Scenario A. Min-Time Arrival}
In this scenario the objective to to identify the optimal control whereby a target, $T$, arrives at a desired terminal location, $F$, in minimum time regardless of the engagement zone. For this optimal control problem, the initial state, $\mathbf{x}_0$, and final state, $\mathbf{x}_f$, are provided. Because the objective cost is Mayer, the Hamiltonian is the following:
\begin{equation}
	\mathscr{H}_A = \left< \mathbf{p}(\mathbf{x}(t),u(t),t),\mathbf{f}(\mathbf{x}(t),u(t),t)\right>,
	\label{eq:HamiltonianA1}
\end{equation}
where the subscript $A$ corresponds to the Hamiltonian for Scenario A. The Hamiltonian from \cref{eq:HamiltonianA1} is the following:
\begin{equation}
	\mathscr{H}_A = p_{x_T} v_T \cos \psi_T + p_{y_T}v_T\sin\psi_T.
\end{equation}
To determine the optimal control, the necessary conditions for optimality are formed. First, consider the stationarity condition:
\begin{equation}
	\frac{\partial \mathscr{H}_A}{\partial \psi_T} = 0 = - p_{x_T} \sin \psi_T + p_{y_T} \cos \psi_T.
\end{equation}
Using algebraic manipulation, the optimal control, $\psi_T^*$, is found to be a function of the costates,
\begin{equation}
	\psi_T^* = \cos^{-1} \begin{pmatrix} \frac{p_{x_T}}{\sqrt{p_{x_T}^2 + p_{y_T}^2}} \end{pmatrix} = \sin^{-1} \begin{pmatrix} \frac{p_{y_T}}{\sqrt{p_{x_T}^2 + p_{y_T}^2}} \end{pmatrix}.
	\label{eq:OptimalControlB}
\end{equation}
The costate equations are obtained by taking the partial of the Hamiltonian with respect to the states. Since the states do not appear explicitly in the Hamiltonian, the costates are constant,
\begin{equation}
	\dot{p}_{x_T} = -\frac{\partial \mathscr{H}_B}{\partial x_T} = 0 \quad \dot{p}_{y_T} = -\frac{\partial \mathscr{H}_B}{\partial y_T} = 0.
\end{equation} 
Because the costates are constant and the optimal control is a function of the costates, the optimal control is constant. Therefore, the optimal control,
\begin{equation}
	\psi_T^* = \tan^{-1}\left( \frac{y_f-y_0}{x_f-x_0}\right).
	\label{eq:OptimalControlA}
\end{equation}
The time of arrival is simply the distance over the speed,
\begin{equation}
	t_{f,A} = \frac{\overline{TF}}{v_T} = \frac{\sqrt{(x_f-x_0)^2+(y_f-y_0)^2}}{v_T}.
\end{equation}

\subsection{Scenario B. Min-Time Arrival while Avoiding the EZ}
In this scenario, the target vehicle strives to reach the desired final location $F$ in minimum time while completely avoiding the EZ. For this problem the initial location ($\mathbf{x}_0$) and final location ($\mathbf{x}_f$) of the target are provided; both are located outside the EZ. Because the objective cost \cref{eq:ObjectiveB} is Mayer, the Hamiltonian is does not contain a Lagrange cost. This results in the following Hamiltonian:
\begin{equation}
	\mathscr{H}_B = \left< \mathbf{p}(\mathbf{x}(t),u(t),t), \mathbf{f}(\mathbf{x}(t),u(t),t) \right>+ \mu c(\mathbf{x}(t),u(t),t),
	\label{eq:HamiltonianB1}
\end{equation} 
where, the subscript $B$ denotes this Hamiltonian corresponds to Scenario B. The Lagrange variable $\mu$ is multiplied by the inequality constraint $c(\mathbf{x}(t),u(t),t)$ from \cref{eq:ineqConstr1}. 
When $c < 0$, the path constraint is satisfied and $\mu = 0$; however, when $c =0$, the path constraint is active and $\mu < 0$~\cite{bryson1975applied}.
The Hamiltonian from \cref{eq:HamiltonianB1}, more explicitly, is
\begin{equation}
	\mathscr{H}_B = p_{x_T} v_T \cos \psi_T + p_{y_T}v_T\sin\psi_T + \mu \left( \rho(\lambda(x_T,y_T)),\xi(x_T,y_T,\psi_T)) - d(x_T,y_T) \right).
	\label{eq:HamiltonianB2}
\end{equation}

\begin{thm}
\label{thm:ScenarioB1}
When the target is not located on the boundary of the EZ the optimal trajectory is a straight-line trajectory, $\psi_T(t)$ is constant.
\end{thm}

\begin{proof}
	When the target is not on the boundary of the EZ then the Lagrange variable $\mu = 0$. Therefore the Hamiltonian from \cref{eq:HamiltonianB2} is
	\begin{equation}
		\mathscr{H}_B = p_{x_T} v_T \cos \psi_T + p_{y_T}v_T\sin\psi_T.
	\end{equation}
	Taking the partial of the Hamiltonian with respect to the control and setting equal to zero, the stationary condition is formed and is,
	\begin{equation}
		0 = \frac{\partial \mathscr{H}_B}{\partial \psi_T}.
		\label{eq:stationaryB}
	\end{equation}
	Evaluation of \cref{eq:stationaryB}, the following is obtained:
	\begin{equation}
	\begin{aligned}
		0 &= - p_{x_T}v_T\sin\psi_T + p_{y_T}v_T\cos\psi_T\\
		  &=- p_{x_T}\sin\psi_T + p_{y_T}\cos\psi_T
	\end{aligned}
	\label{eq:evalStationaryB}
	\end{equation}
	Algebraically manipulating \cref{eq:evalStationaryB}, the optimal control is found to be a function of the costates,
	\begin{equation}
		\psi_T^* = \cos^{-1} \begin{pmatrix} \frac{p_{x_T}}{\sqrt{p_{x_T}^2 + p_{y_T}^2}} \end{pmatrix} = \sin^{-1} \begin{pmatrix} \frac{p_{y_T}}{\sqrt{p_{x_T}^2 + p_{y_T}^2}} \end{pmatrix}.
		\label{eq:control1}
	\end{equation}
	Taking the partial of the Hamiltonian with respect to the states provides the dynamics of the costates. Since the states $x_T$ and $y_T$ do not appear explicitly in the Hamiltonian (under the assumption that the path constraint is not active), the costates are constant,
	\begin{equation}
		\dot{p}_{x_T} = -\dfrac{\partial \mathscr{H}_B}{\partial x_T} = 0 \quad \dot{p}_{y_T} = -\dfrac{\partial \mathscr{H}_B}{\partial y_T} = 0.
		\label{eq:costateDyn1}
	\end{equation}		
	The costate are constant by \cref{eq:costateDyn1} and the optimal control is solely a function of the costates as shown in \cref{eq:control1}. Therefore, the optimal control, $\psi_T^*$ is constant when the constraint is not active. As a result, the optimal trajectory is a straight line trajectory whenever $c < 0$.
\end{proof}

\begin{thm}
\label{thm:ScenarioB2}
When the target is constrained along the boundry, $\rho_\mathrm{max} = d$ and therefore the	heading which keeps the target on the boundary of the EZ is:
\begin{equation}
	\psi_T^* = 2 \tan^{-1} \left(  \frac{-(\sigma_1+ x_T) \pm \sqrt{y_T^2 - (\sigma_1 - x_T)(\sigma_1+x_T)}}{(\sigma_1 - x_T)} \right),
\end{equation}
where
\begin{equation*}
	\sigma_1 = \frac{2(x_T^2+y_T^2) - R_\mathrm{max}\sqrt{x_T^2+y_T^2}}{R_\mathrm{max}}.
\end{equation*}
\end{thm}
\begin{proof}
	When the target is constrained on the boundary of the EZ, $\rho_\mathrm{max} = d$, and therefore:
	\begin{equation}
		\frac{R_\mathrm{max}}{2}(\cos \xi + 1) = \sqrt{x_T^2+y_T^2}.
		\label{eq:rhoEqD1}
	\end{equation}
	recall, $\xi = \psi_T - \lambda - \pi$ and $\lambda = \tan^{-1}(y_T/x_T)$. Making this substitution in \cref{eq:rhoEqD1},
	\begin{equation}
		\frac{R_\mathrm{max}}{2}(\cos \left(\psi_T - \lambda - \pi \right) + 1) = \sqrt{x_T^2+y_T^2}.
		\label{eq:rhoEqD2}
	\end{equation}
	Simplifying	\cref{eq:rhoEqD2} using the identities: $\cos(x-\pi) = -\cos(x)$ and $\cos(a-b) = \sin(a)\sin(b)-\cos(a)\cos(b)$.
	\begin{equation}
	\begin{aligned}
		\frac{R_\mathrm{max}}{2}(1 - \cos \left(\psi_T - \lambda\right)) &= \sqrt{x_T^2+y_T^2},\\
		\frac{R_\mathrm{max}}{2}(1 - \sin\psi_T\sin\lambda - \cos\psi_T \cos\lambda) &= \sqrt{x_T^2+y_T^2},
	\end{aligned}
		\label{eq:rhoEqD3}
	\end{equation}
	As an aside:
	\begin{equation}
	\begin{aligned}
		\cos\lambda &= \cos\left(\tan^{-1}\left(\frac{y_T}{x_T}\right)\right) = \frac{1}{\sqrt{\frac{y_T^2}{x_T^2}+1}} = \frac{x_T}{\sqrt{x_T^2+y_T^2}},\\
		 \sin\lambda &= \sin\left(\tan^{-1}\left(\frac{y_T}{x_T}\right)\right) = \frac{y_T}{x_T\sqrt{\frac{y_T^2}{x_T^2}+1}} = \frac{y_T}{\sqrt{x_T^2+y_T^2}}.
	\end{aligned}
		\label{eq:atanLaws}
	\end{equation}
	Substituting \cref{eq:atanLaws} into \cref{eq:rhoEqD3}, the following is obtained:
	\begin{equation}
		\frac{R_\mathrm{max}}{2}\left( 1 - \frac{x_T \cos \psi_T}{\sqrt{x_T^2+y_T^2}} -\frac{y_T \sin \psi_T}{\sqrt{x_T^2+y_T^2}}\right) = \sqrt{x_T^2+y_T^2}.
		\label{eq:rhoEqD4}
	\end{equation}
    Simplifying \cref{eq:rhoEqD4} by re-arranging,
	\begin{equation}
		x_T \cos \psi_T + y_T \sin \psi_T - \frac{R_\mathrm{max}\sqrt{x_T^2+y_T^2} -2 (x_T^2+y_T^2)}{R_\mathrm{max}}= 0.
	\end{equation}
	Defining $\sigma_1$ as,
	\begin{equation}
		\sigma_1 = \frac{2(x_T^2+y_T^2) - R_\mathrm{max}\sqrt{x_T^2+y_T^2}}{R_\mathrm{max}},
		\label{eq:sigma1}
	\end{equation}
	the following is obtained:
	\begin{equation}
		x_T \cos \psi_T + y_T \sin \psi_T + \sigma_1 = 0.
		\label{eq:simpleRhoD}
	\end{equation}
	Letting $\tau = \tan(\psi_T/2)$. Then $\sin(\psi) = \tfrac{2 \tau}{ (\tau^2+1)}$ and $\cos \psi_T = \tfrac{1-\tau^2}{\tau^2+1}$. Substitution of these half-angle identities in \cref{eq:simpleRhoD}, the following is obtained:
	\begin{equation}
		\sigma_1 + \frac{x_T}{\tau^2+1} + \frac{2y_T\tau}{\tau^2+1} - \frac{x_T \tau^2}{\tau^2+1}= 0
		\label{eq:halfAng1}
	\end{equation}	
	Using the common denominator $\tau^2+1$, \cref{eq:halfAng1} is simplified to be:
	\begin{equation}
		x_T + \sigma_1 + 2 y_2 \tau + (\sigma_1 - x_T) \tau^2 = 0
		\label{eq:quadraticStep}
	\end{equation}
	Using the quadratic equation, $\tau$ is obtained:
	\begin{equation}
		\tau = \tan (\psi_T/2) = \frac{-(\sigma_1+ x_T) \pm \sqrt{(2y_T)^2 - 4(\sigma_1 - x_T)(\sigma_1+x_T)}}{2(\sigma_1 - x_T)}
		\label{eq:quadraticStep2}
	\end{equation}
	Finally, simplifying \cref{eq:quadraticStep2}, the optimal control, $\psi_T^*$, is obtained and is
	\begin{equation}
		\psi_T^* = 2 \tan^{-1} \begin{pmatrix}  \dfrac{-(\sigma_1+ x_T) \pm \sqrt{y_T^2 - (\sigma_1 - x_T)(\sigma_1+x_T)}}{(\sigma_1 - x_T)} \end{pmatrix}.
		\label{eq:OptimControlB}
	\end{equation}
\end{proof}

Below are some important features of the solution which may be deduced directly from Lemmas~\ref{thm:ScenarioB1} and~\ref{thm:ScenarioB2}.
	The constrained control keeps the target on the border of the engagement zone, until a tangent exists for the target to depart the engagement zone and reach the desired final location.
%
	Provided the initial location and final desired location of the target reside strictly outside the EZ at their respective times, the optimal target trajectory for Scenario B, assuming the constraint $c$ will become active, is comprised of a sequence of trajectories: unconstrained, constrained, then unconstrained ($UCU$). The optimal unconstrained arcs are straight lines and the optimal control for constrained arc obeys \cref{eq:OptimControlB}.
%
	The time whereby the target arrives at the desired location for Scenario B will be greater than or equal to the arrival time from Scenario A; namely, $t_{f,B} \geq t_{f,A}$.
%
	The optimal control from Scenario A and B are identical in the event that a straight line trajectory from T to F does not intersect the EZ. Because a straight line does not intersect the EZ, the Lagrange variable, $\mu$ is zero for the entire trajectory; by \Cref{thm:ScenarioB1} the trajectory is a straight line and therefore the optimal control follows \cref{eq:OptimalControlA}.

\subsection{Scenario C. Min-Time Arrival while Distancing Itself from the EZO}

In this scenario the objective is to identify the optimal control whereby the target, $T$, arrives at a desired terminal location, $F$, while distancing itself from the origin of the engagement zone. The objective cost is defined in \cref{eq:ObjectiveC}. The objective is a Bolza problem and has penalties in final time and incurs a running cost in the event that the target enters the engagement zone. The Hamiltonian for this optimal control problem is of the form:
\begin{equation}
	\mathscr{H}_C = \left< \mathbf{p}(\mathbf{x}(t),u(t),t),\mathbf{f}(\mathbf{x}(t),u(t),t) \right> + l(\mathbf{x}(t),u(t),t).
	\label{eq:HamiltonianC1}
\end{equation}
The subscript $C$ denotes that the Hamiltonian corresponds to Scenario C. The Hamiltonian in \cref{eq:HamiltonianC1} is more explicitly,
\begin{equation}
	\mathscr{H}_C = \begin{cases}
 	p_{x_T} v_T \cos \psi_T + p_{y_T}v_T\sin\psi_T + \dfrac{\rho}{d} - 1 & d \leq \rho \\
 	p_{x_T} v_T \cos \psi_T + p_{y_T}v_T\sin\psi_T & d > \rho
 \end{cases}
\end{equation}
Recall, the range of the engagement zone,  $\rho$, from \cref{eq:EZ}, the relative bearing angle, $\xi$, from \cref{eq:aspectAngle}, and the line of sight angle from \cref{eq:LOS}, repeated for the reader's convenience:
\begin{equation*}
\begin{aligned}
	\rho_{\mathrm{max}}(\xi) = \frac{R_{\mathrm{max}}}{2} \left(\cos \xi + 1\right), \quad 
	\xi = \psi_T - \lambda - \pi, \quad
	\lambda &= \tan^{-1}(y_T/x_T)
\end{aligned}
\end{equation*}
Since the Hamiltonian, where $d>\rho$, is the same as $\mathscr{H}_A$; the case where the target is outside the EZ results in straight line optimal trajectories. For the case where the target is inside the EZ, the Hamiltonian,
\begin{equation}
	\mathscr{H}_C = p_{x_T} v_T \cos \psi_T + p_{y_T}v_T\sin\psi_T + \dfrac{\rho_\mathrm{max}}{d} - 1, \quad d \leq \rho.
	\label{eq:HamiltonianC2}
\end{equation}
Substitution of \cref{eq:EZ,eq:aspectAngle,eq:LOS} into \cref{eq:HamiltonianC2} the following is obtained,
\begin{equation}
	\mathscr{H}_C = p_{x_T} v_T \cos \psi_T + p_{y_T}v_T\sin\psi_T + \dfrac{R_\mathrm{max}}{2(x_T^2+y_T^2)}\left( \sqrt{x_T^2+y_T^2} - x_T \cos \psi_T - y_T \sin \psi_T \right) - 1
	\label{eq:HamiltonianC3}
\end{equation}

The necessary conditions for optimality are obtained by taking partials of the Hamiltonian with respect to the control and the states. First, consider the dynamics of the costates, obtained by taking the partial of the Hamiltonian $\mathscr{H}_C$ in \cref{eq:HamiltonianC3} with respect to the states,
\begin{equation}
\begin{aligned}
	\dot{p}_{x_T} &= -\frac{\partial \mathscr{H}_C}{\partial x_T} = - \frac{\partial }{\partial x_T} \left( \dfrac{R_\mathrm{max}}{2(x_T^2+y_T^2)}\left( \sqrt{x_T^2+y_T^2} - x_T \cos \psi_T - y_T \sin \psi_T \right) \right),\\ 
	\dot{p}_{y_T} &= -\frac{\partial \mathscr{H}_C}{\partial y_T} = - \frac{\partial }{\partial y_T} \left( \dfrac{R_\mathrm{max}}{2(x_T^2+y_T^2)}\left( \sqrt{x_T^2+y_T^2} - x_T \cos \psi_T - y_T \sin \psi_T \right) \right). 
\end{aligned}
\label{eq:costatesC1}
\end{equation}
The partials are computed and are as follows:
\begin{equation}
	\begin{aligned}
	\dot{p}_{x_T} &=  \frac{ R_\mathrm{max} x_T}{2 \left(x_T^2+y_T^2\right)^2}\left( \sqrt{x_T^2+y_T^2} - x_T \cos \psi_T - y_T \sin \psi_T\right) - \frac{R_\mathrm{max}}{2(x_T^2+y_T^2)}\begin{pmatrix} \dfrac{x_T}{\sqrt{x_T^2+y_T^2}} - \cos \psi_T\end{pmatrix}\\
	\dot{p}_{y_T} &=  \frac{ R_\mathrm{max} y_T}{2 \left(x_T^2+y_T^2\right)^2}\left( \sqrt{x_T^2+y_T^2} - x_T \cos \psi_T - y_T \sin \psi_T\right) - \frac{R_\mathrm{max}}{2(x_T^2+y_T^2)}\begin{pmatrix} \dfrac{y_T}{\sqrt{x_T^2+y_T^2}} - \sin \psi_T\end{pmatrix}\\
	\end{aligned}
	\label{eq:partialsC}
\end{equation}
Using the partials in \cref{eq:partialsC} the costate dynamics are known. Next, the stationarity condition provides the optimal control as a function of the states and costates:
\begin{equation}
	\frac{\partial \mathscr{H}_C}{\partial \psi_T} = - p_{x_T}v_T \sin \psi_T + p_{y_T}v_T\cos \psi_T + \frac{R_\mathrm{max}}{2(x_T^2+y_T^2)}\left( x_T \sin \psi_T - y_T \cos \psi_T \right) = 0
	\label{eq:StationaryC}
\end{equation}
Solving \cref{eq:StationaryC} for the optimal control, $\psi_T^*$, the following is obtained:
\begin{equation}
	\psi_T^* = \cos^{-1}\left( \dfrac{\dfrac{R_\mathrm{max}x_T}{2(x_T^2+y_T^2)}- p_{x_T}v_T}{v_T^2(p_{x_T}^2 + p_{y_T}^2) + \dfrac{R_\mathrm{max}}{4(x_T^2+y_T^2)}-\dfrac{v_TR_\mathrm{max}}{(x_T^2+y_T^2)}\left( x_T p_{x_T} + y_T p_{y_T}\right)} \right)
	\label{eq:OptimControlC1}
\end{equation}
Simplifying \cref{eq:OptimControlC1} using algebraic manipulation, the following is obtained:
\begin{equation}
	\psi_T^* = \cos^{-1}\left( \dfrac{2 R_\mathrm{max}x_T - 4 p_{x_T}(x_T^2+y_T^2)}{4v_T^2(x_T^2+y_T^2)(p_{x_T}^2+p_{y_T}^2) + R_\mathrm{max}^2 - 4v_T R_\mathrm{max}(x_Tp_{x_T}+y_Tp_{y_T}) }\right)
	\label{eq:OptimControlC2}
\end{equation}
Note, that the optimal control $\psi_T^*$ for the entire trajectory for Scenario C is constant when the target is outside the EZ and satisfies \cref{eq:OptimControlC2} when the target is inside the EZ. The switching time wherein the target enters and exits the engagement zone is unknown and can be obtained numerically. The necessary conditions for optimality for Scenario C are defined in \cref{eq:EOM,eq:partialsC,eq:OptimControlC2}. These equations paired with the initial state, $T = \mathbf{x}(t_0) = \mathbf{x}_0$ and final state, $F = \mathbf{x}(t_f) = \mathbf{x_f}$ define a two-point boundary-value problem (TPBVP). Since the costates are unknown at the initial and final time, this TPVBP is solved in this paper numerically, rather than analytically, using a nonlinear program (NLP) solver. It should be noted, as outlined in \Cref{tab:OCPs}, that the final time is free.

\begin{rmk}
	The final time for Scenario C is bounded by the final time from Scenario A: the straight line trajectory, and Scenario B: the trajectory which does not allow the target to enter the engagement zone:
	\begin{equation*}
		t_{f,C} \in [t_{f,A},t_{f,B}).
	\end{equation*}	
\end{rmk} 

\subsection{Specified Time of Arrival while Distancing Itself from the EZO}

In this scenario, the objective is to identify the optimal control wherein the target, $T$, arrives at a desired terminal location, $F$, at a required time, $t_\mathrm{go} \in [t_{f,A},t_{f,B})$, while distancing itself from the EZO. The penalty function for the target vehicle is defined in \cref{eq:penaltyFunction}. The objective cost function which the target strives to minimize is defined in \cref{eq:ObjectiveD}. The objective is a Bolza problem and has penalties in final time and incurs a running cost in the event that the target enters the engagement zone. 

\begin{equation}
	\mathscr{H}_D = \left< \mathbf{p}(\mathbf{x}(t),u(t),t),\mathbf{f}(\mathbf{x}(t),u(t),t) \right> + l(\mathbf{x}(t),u(t),t).
	\label{eq:HamiltonianD1}
\end{equation}
This scenario is identical to Scenario C for the exception that the final time is specified rather than free; namely,
\begin{equation}
	\begin{aligned}
		\mathbf{x}(t_\mathrm{go}) = F.
	\end{aligned}
\end{equation}
From the preceding analysis, one may infer that the target's trajectories which lie outside the engagement zone are straight lines, and those which pass through the engagement zone obey the necessary conditions for optimality described in \cref{eq:OptimControlC2,eq:partialsC,eq:EOM}. The TPBVP which develops is one with fixed initial state, fixed final state, fixed initial time, and fixed final time. Two unknowns remain: the time where the target enters the EZ and the time where the target leaves the EZ. When the specified time $t_\mathrm{go} = t_{f_,A}$, the optimal target trajectory is a straight-line trajectory and the optimal control problem from Scenario D is identical to Scenario A. 

\section{Results}
\label{sec:results}

The numerical solution approach taken here is based on direct transcription~\cite{betts2010practical, kelly2017introduction} in lieu of indirect optimization methods such as multiple shooting.
More specifically, we utilize PSM~\cite{fahroo2005pseudospectral} which relies on quadrature techniques to perform the integrations needed to evaluate the cost functionals (specifically for Scenarios C and D).
Here, the Legendre-Gauss-Radau (LGR) nodes are utilized, and we set $M=19$, where $M$ is the number of collocation nodes.
Since the LGR scheme places a node at the start of the trajectory, but not at the end, we add an additional node at the trajectory's endpoint in order to enforce terminal state constraints.
The choice of PSM over multiple shooting is in order to have a unified implementation approach amongst the four scenarios A--D; multiple shooting would have entailed more tailored numerical approaches to handle treating the EZ as a hard constraint versus soft constraint.

Figure~\ref{fig:psmFlow} describes the numerical solution procedure.
We utilize the generic constrained nonlinear solver algorithm Constrained Optimization by Linear Approximation (COBYLA)~\cite{powell1994direct} which is a local, derivative-free algorithm implemented within the NLopt software package~\cite{johnson2019nlopt} (utilized from within the Julia programming language).

\begin{figure}[H]
\centering
	\includegraphics[width=3.5in]{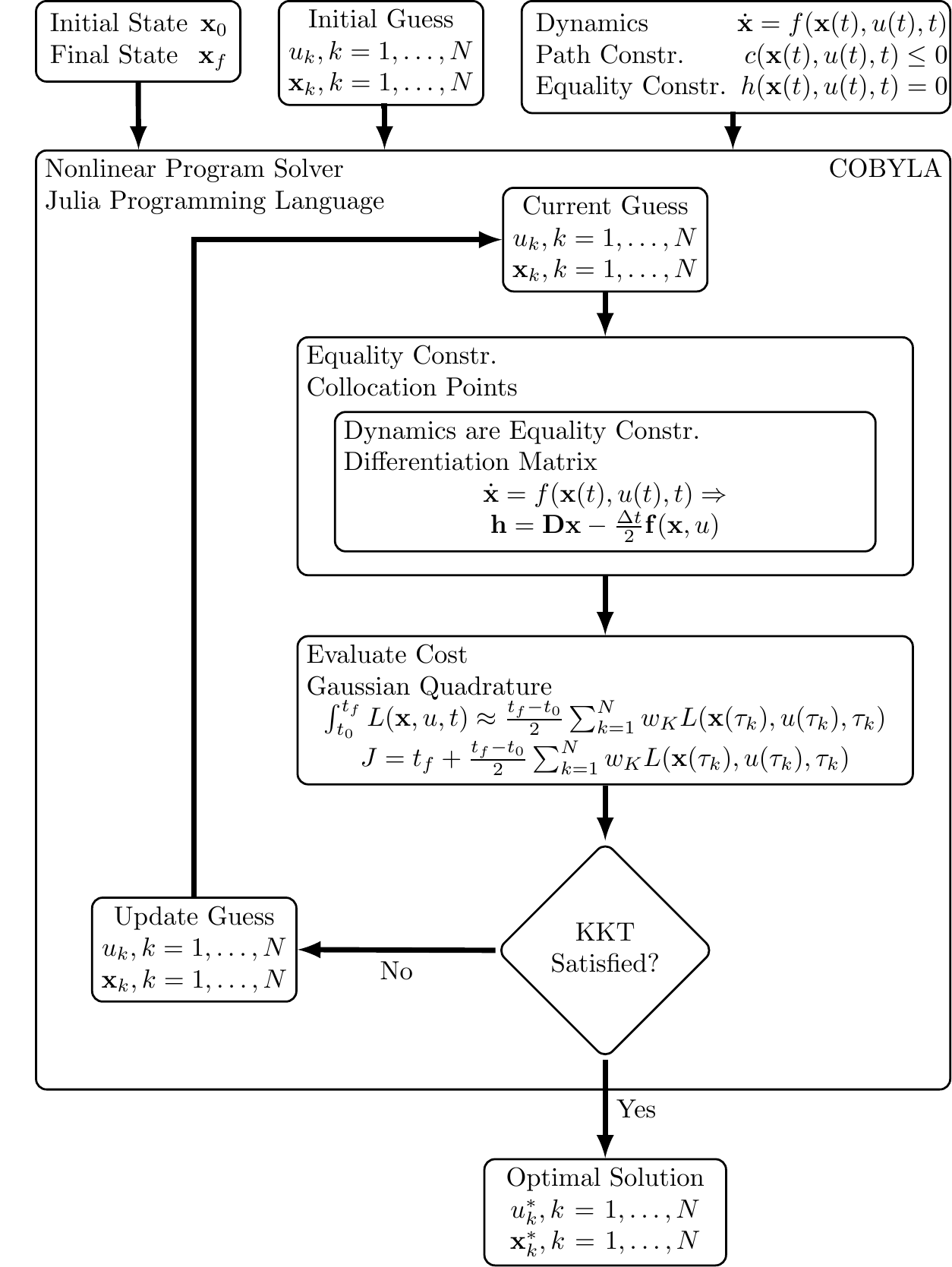}
	\caption{A flow diagram for implementing pseudospectral methods for solving optimal control problems directly.}
	\label{fig:psmFlow}
\end{figure}

Figure~\ref{fig:results} contains sample solution trajectories for the case where $\mathbf{x}_0 = (1,3)^\top$ and $\mathbf{x}_f = (-0.5, -3)^\top$, $v_T = 1$, $R_\mathrm{max} = 2$, $R_\mathrm{min} = 0$, and various $k_{EZ}$.
The effect of varying $k_{EZ}$ is clearly demonstrated in Fig.~\ref{fig:C_results} whereupon smaller $k_{EZ}$ result in trajectories that approach the minimum-time path (c.f.~trajectory A in Fig.~\ref{fig:all_results}) and larger $k_{EZ}$ result in trajectories that approach the fully constrained path (c.f.~trajectory B in Fig.~~\ref{fig:all_results})
For Scenario D, the selected desired final time is chosen such that $t_{f,D} = 6.25 \in \left[ t_{f,A}, t_{f,B} \right] $.
The cardioid-shaped EZ is drawn for Scenario B when the constraint is active (i.e., $d = \rho_\mathrm{max}$) and also for Scenario C when $d < \rho_\mathrm{max}$.

\begin{figure}[H]
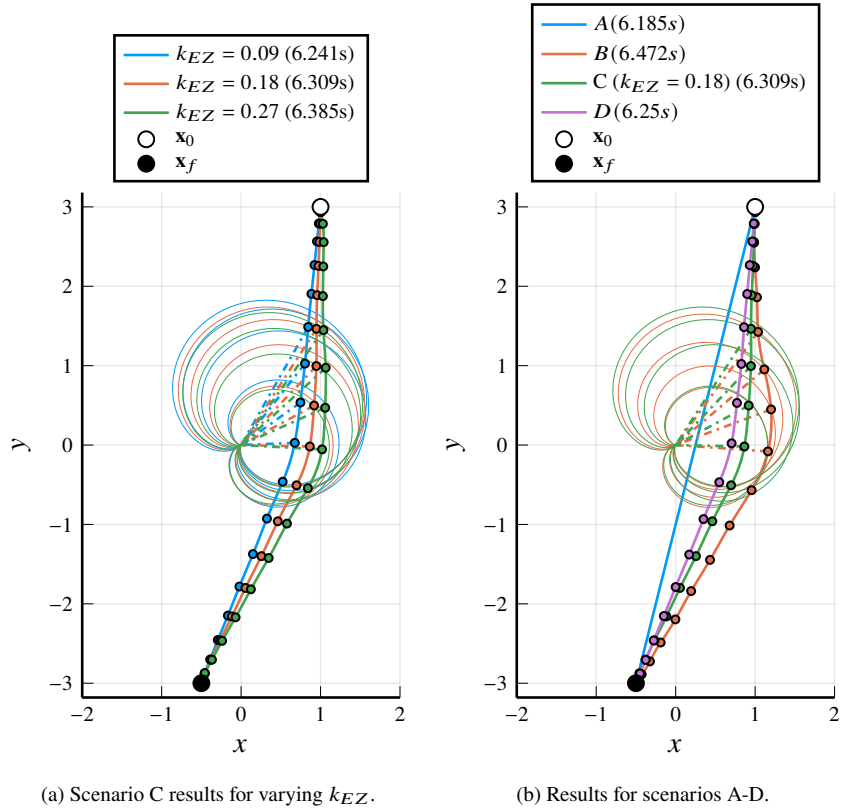

	\centering
	\pgfplotsset{unit vector ratio*=1 1 1}
	\subfloat[Scenario C results for varying $k_{EZ}$.]{\includegraphics[width=.35\textwidth]{scenarioC_results.tikz}\label{fig:C_results}}
	\subfloat[Results for scenarios A-D.]{\includegraphics[width=.35\textwidth]{all_results.tikz}\label{fig:all_results}}
	\caption{
		Various solutions for $\mathbf{x}_0 = (1,3)^\top$ and $\mathbf{x}_f=(-0.5, -3)^\top$.
		The instantaneous EZ is shown when the constraint is active for Scenario B and when $g_{EZ} \neq 0$ for Scenario C.
	}%
	\label{fig:results}
\end{figure}

\section{Conclusions}
\label{sec:conclusions}

In this work we presented a constrained trajectory optimization problem with potential applications to Urban Air Mobility, Unmanned Traffic Management, as well as adversarial scenarios.
The dynamic keepout zone, or EZ, was assumed to be a cardiod shape which is always oriented towards the target vehicle thereby making our analysis a worst-case analysis.
We demonstrated the effect of various interpretations of the EZ: namely, as a hard constraint (the vehicle \textit{must not} enter it), as a soft constraint, and even as the cost functional of interest.
The first order necessary conditions for optimality were derived, and useful features of the overall solution were extracted.
For example, when the EZ is a hard constraint, it was shown that the optimal trajectory, in cases where the constraint becomes active, is comprised of a sequence of unconstrained, constrained, and unconstrained arcs; the target vehicle's heading is constant along the unconstrained arcs, and the optimal constrained control was derived in closed form.
A generic solution procedure based on pseudo-spectral methods was utilized to solve for trajectories in each of the cases.
Future work will focus on further exploiting the first order necessary conditions for optimality to develop more tailored numerical solution procedures.
Additionally, a state-feedback controller which approximates the optimal trajectory, preferably with a bound on worst-case performance w.r.t.~the optimal, is desired.

\section{Acknowledgments}

The views expressed in this document are those of the author and do not reflect the official policy or position of the United States Air Force, the United States Department of Defense or the United States Government. This work has been supported in part by AFOSR LRIR No. 21RQCOR084. DISTRIBUTION STATEMENT A. Approved for public release: distribution unlimited. (AFRL-2021-1507).

\bibliography{library.bib}

\end{document}